\documentclass[12pt,reqno]{amsart}
\usepackage[margin=.95in]{geometry}
\usepackage[english]{babel}
\usepackage[utf8]{inputenc}
\usepackage{subcaption}
\usepackage{amsmath}
\usepackage{amssymb}
\usepackage{amsfonts}
\usepackage{amsthm}
\usepackage{mathrsfs}
\usepackage[all]{xy}
\usepackage[pdftex]{graphicx}
\usepackage{calc}
\usepackage{color}
\usepackage{cite}
\usepackage{ifthen}
\usepackage[T1]{fontenc}
\usepackage{url}
\usepackage{indent first}
\usepackage[labelfont=bf,labelsep=period,justification=raggedright]{caption}
\usepackage[english]{babel}
\usepackage[utf8]{inputenc}
\usepackage{hyperref}
\usepackage[colorinlistoftodos]{todonotes}
\usepackage{tkz-fct}
\usepackage{tikz}
\usetikzlibrary{calc}
\usepackage{multicol}
\PassOptionsToPackage{dvipsnames,svgnames}{xcolor}
\usepackage{textcomp}
\usepackage{bbm}
\usepackage{ stmaryrd }
\usepackage{xcolor}
\usepackage{xkeyval}

\DeclareMathOperator{\lcm}{lcm}

\DeclareMathOperator{\ord}{ord}

\newcommand{\NN}{\mathbb{N}}

\newcommand{\ZZ}{\mathbb{Z}}

\def\multiset#1#2{\ensuremath{\left(\kern-.3em\left(\genfrac{}{}{0pt}{}{#1}{#2}\right)\kern-.3em\right)}}

\theoremstyle{plain}
\newtheorem{thm}{Theorem}
\newtheorem{lemma}[thm]{Lemma}
\newtheorem{cor}[thm]{Corollary}
\newtheorem{conj}[thm]{Conjecture}
\newtheorem{prop}[thm]{Proposition}

\newtheorem{qn}[thm]{Question}

\theoremstyle{definition}
\newtheorem{defn}[thm]{Definition}

\theoremstyle{remark}
\newtheorem{rem}[thm]{Remark}
\newtheorem{notation}[thm]{Notation}
\newtheorem*{ex}{Example}

\numberwithin{equation}{section}
\numberwithin{thm}{section}

\begin{document}

\title{On Generalized Carmichael Numbers}
\author{Yongyi Chen \and Tae Kyu Kim}
\date{\today}
\maketitle

\begin{abstract}
    Given an integer $k$, define $C_k$ as the set of integers $n > \max(k,0)$ such that $a^{n-k+1} \equiv a \pmod{n}$ holds for all integers $a$. We establish various multiplicative properties of the elements in $C_k$ and give a sufficient condition for the infinitude of $C_k$. Moreover, we prove that there are finitely many elements in $C_k$ with one and two prime factors if and only if $k>0$ and $k$ is prime. In addition, if all but two prime factors of $n \in C_k$ are fixed, then there are finitely many elements in $C_k$, excluding certain infinite families of $n$. We also give conjectures about the growth rate of $C_k$ with numerical evidence. We explore a similar question when both $a$ and $k$ are fixed and prove that for fixed integers $a \geq 2$ and $k$, there are infinitely many integers $n$ such that $a^{n-k} \equiv 1 \pmod{n}$ if and only if $(k,a) \neq (0,2)$ by building off the work of Kiss and Phong. Finally, we discuss the multiplicative properties of positive integers $n$ such that Carmichael function $\lambda(n)$ divides $n-k$.
\end{abstract}

\section{Introduction}
In 1860, Fermat proved that if $p$ is a prime number, then $p$ divides $a^{p-1} - 1$ for any integer $a$ not divisible by $p$, a result now known as \emph{Fermat's little theorem}. An equivalent formulation is that whenever $p$ is a prime number, $p$ divides $a^p - a$ for all integers $a$. Fermat's little theorem gives a test for prime numbers. For a given integer $n$, randomly pick some integers $a_1, a_2, a_3, \ldots, a_l$. If $a_i^n \equiv a_i \pmod{n}$ is true for all $1 \leq i \leq l$, then it is probable that $n$ is a prime.

The question now arose for the converse of Fermat's theorem: if an integer $n$ is such that $a^n \equiv a$ (mod $n$) for all integers $a$, then is $n$ necessarily prime? A counterexample to the converse was found in 1910 by Carmichael \cite{carmichael1910note} who noted that $561 = 3 \times 11 \times 17$ divides $a^{561} - a$ for all integers $a$. This was proven using a result by Korselt \cite{korselt1899probleme} who gave an equivalent condition for a positive integer $n$ to divide $a^n - a$ for all integers $a$.

In a series of papers around 1910, Carmichael extensively studied the composite integers with this property, which have become to be known as the \emph{Carmichael numbers}. In 1912, Carmichael \cite{carmichael1912composite} described an algorithm to construct these numbers, wishfully stating that ``this list might be indefinitely extended.'' This conjecture was resolved in 1994 when Alford, Granville, and Pomerance \cite{alford1994there} proved that there are infinitely many Carmichael numbers, inspired by a heuristic argument by Erd\H{o}s.

In 1939, Chernick \cite{chernick1939fermat} proved that if $6n+1, 12n+1, 18n+1$ are all primes, then $(6n+1)(12n+1)(18n+1)$ is a Carmichael number. Dickson's conjecture states that for a finite set of linear forms $a_1+b_1n$, $a_2+b_2n$, \ldots, $a_r + b_r n$ with $b_i \geq 1$, then there are infinitely many positive integers $n$ for which they are all prime, unless there is a congruence condition preventing this. A proof of Dickson's conjecture for $r=3$ would prove that there are infinitely many Carmichael numbers of the form $(6n+1)(12n+1)(18n+1)$. It remains an open question whether there are infinitely many Carmichael numbers with a given number of prime factors.

A crucial fact in Carmichael's proof that $561$ is a Carmichael number was Korselt's criterion, which enables us to easily determine whether an integer $n$ is a Carmichael number from its prime factorization.

\begin{prop} [Korselt's criterion]
A positive integer $n$ divides $a^n - a$ for all numbers $a \in \NN$ if and only if $n$ is squarefree and $p-1$ divides $n-1$ for all primes $p$ dividing $n$.
\end{prop}
From Korselt's criterion, we can check that $561$ is a Carmichael number since it is squarefree, and $3-1=2$, $11-1 = 10$, and $17 - 1 = 16$ divide $561 - 1 = 560 = 2^4 \times 5 \times 7$.

We can restate Korselt's criterion using the Carmichael function.
\begin{defn}
The \emph{Carmichael function} $\lambda(n)$ is the greatest order of an element in $(\ZZ/n\ZZ)^\times$.
\end{defn}
It is well-known that the Carmichael function $\lambda(n)$ can be expressed in terms of the prime factors of $n$:
\begin{prop}
Let $n$ be a positive integer with prime factorization $\prod_{i=1}^r p_i^{e_i}$. The Carmichael function has the closed form $\lambda(n) = \lcm_{i=1}^r \left\{\lambda\left(p_i^{e_i}\right)\right\}$ where 
\[ \lambda(p^{e}) = \begin{cases}
p^{e-1}(p-1) & \text{ if $p$ is odd or $p^e = 2, 4$} \\
p^{e-2}(p-1) & \text{ if $p = 2$ and $e > 2$} 
\end{cases}.\] 
In particular, when $n$ is squarefree, $\lambda(n) = \lcm_{i=1}^r \{p_i-1\}$.
\end{prop}
Thus, the following is an equivalent formulation of Korselt's criterion:
\begin{prop}
A positive integer $n$ divides $a^n - a$ for all numbers $a \in \NN$ if and only if $n$ is squarefree and $\lambda(n)$ divides $n-1$.
\end{prop}

In this paper, we first consider the following question:
\begin{qn} \label{qn:main}
Given an integer $k$, for what integers $n > \max(k, 0)$ is $a^{n-k+1} \equiv a \pmod{n}$ for all integers $a$?
\end{qn}
\begin{notation}
We denote the set of positive integers $n$ that satisfy the condition in Question \ref{qn:main} by $C_k$:
\[ C_k = \{ n \in \ZZ : n > \max(k,0) \text{ and } a^{n-k+1} \equiv a \pmod{n} \text{ for all integers } a \}. \]
\end{notation}

Question \ref{qn:main} is a natural generalization of the question about the converse of Fermat's little theorem, which is obtained for $k=1$. Consequently, $C_1$ consists of all of the prime numbers and Carmichael numbers.

While writing up this paper, we discovered that Question \ref{qn:main} was explored in 1999 by Halbeisen and Hungerb{\"u}hler \cite{halbeisen1999generalised}. Many of our results on this question have been previously derived in that paper. However, since some of the results are new and some other results have different proofs, we have included most of our results, both old and new, with our own proofs.

In Subsection \ref{subsec:kgreater0}, we show that for squarefree $k>0$, $kp \in C_k$ for all primes $p \equiv 1 \mod \left( \frac{\lambda(k)}{\gcd(\lambda(k), k)}\right)$. Through numerical testing, we observe elements in $C_k$ that are not of this form. 
\begin{defn} \label{defn:generalizedcarmichael}
For squarefree $k>0$, we define the \emph{generalized Carmichael numbers} as
\[ N_k = C_k \setminus \left\{kp : p \equiv 1 \mod \left( \frac{\lambda(k)}{\gcd(\lambda(k), k)}\right) \right\}. \]
For the other values of $k$, we let $N_k = C_k$.
\end{defn}
Note that $N_1$ coincides with the set of Carmichael numbers. We interpret $N_k$ as a type of generalized Carmichael numbers since they do not exhibit a nice pattern with respect to the set of primes.

Next, we consider the growth rate of the counting functions
\[ C_k(X) = \lvert C_k \cap (0, X] \rvert \]
and
\[ N_k(X) = \lvert N_k \cap (0, X] \rvert. \]

From graphing $N_k(10^7)$ against $N_{-k}(10^7)$ for squarefree $2 \leq k \leq 1000$ in Figure \ref{fig:NvsN}, we give the following conjecture:
\begin{conj} \label{conj:NkvN-k}
    For all integers $k$,
    \[ \lim_{X \to \infty} \frac{N_{-k}(X)}{N_k(X)} = 1. \]
\end{conj}

We also observe that for the same value of $X$, $N_k(X)$ tends to be bigger when $k$ is squarefree and has many prime factors. Unfortunately, we do not give any quantitative statements on this phenomenon.

Erd\H{o}s \cite{erdos1956pseudoprimes} conjectured that the number of Carmichael numbers less than $X$ is $X^{1-o(1)}$. More precisely, from a series of assumptions, he conjectured that $N_1(X)$ should be asymptotically
\[ X \exp\left(-\frac{\log X \log\log\log X}{\log \log X} \right). \]

\begin{defn} \label{defn:dcoefficient}
Extending the original definition by Pomerance, Selfridge, and Wagstaff \cite{pomerance1980pseudoprimes}, we define $d_k(X)$ by
\[ N_k(X) = X \exp\left(-d_k(X) \frac{\log X \log\log\log X}{\log \log X} \right). \]
\end{defn}
Pomerance, Selfridge, and Wagstaff \cite{pomerance1980pseudoprimes} proved that $\liminf d_1(X) \geq 1$ and suggested that $\limsup d_1(X)$ might be $2$. Pomerance \cite{pomerance1981distribution,pomerance1989two} gave a heuristic argument suggesting that $\lim d_1(X) = 1$. Pinch \cite{pinch1993carmichael} describes the case for $k=1$ in more detail.

By graphing $d_k(10^7)$ for squarefree $-1000 \leq k \leq 1000$ in Figure \ref{fig:dvsd}, we find that $d_k(10^7)$ for these $k$ lie between $1$ and $2$. Unfortunately, we were unable to draw any conclusions about the limit $\lim_{X \to \infty} d_k(X)$ from our numerical tests.

Our second question comes from fixing the values of $a$ and $k$ in the congruence condition $a^{n-k} \equiv 1 \pmod{n}$. In general, fixing the value of $a$ enables more positive integers $n$ to satisfy the congruence condition. We ask the following question:
\begin{qn} \label{qn:fixeda}
Let $a \geq 2$ and $k$ be fixed integers. Do there exist infinitely many composite integers $n > \max(k,0)$ such that $a^{n-k} \equiv 1 \pmod{n}$?
\end{qn}

In 1970, Rotkiewicz \cite{rotkiewicz1970pseudoprime} asked about whether, for arbitrary fixed integers $a \geq 2$ and $k > 0$, there are infinity many positive integers $n$ such that $a^{n-k} \equiv 1 \pmod{n}$. The answer is affirmative in the case $k=1$; the positive composite integers $n$ such that $a^{n-1} \equiv 1 \pmod{n}$ are called \emph{pseudoprimes} to base $a$. In 1963, Makowski \cite{makowski1962generalization} proved that for $k>0$, there are infinitely many composite integers $n$ such that $a^{n-k} \equiv 1 \pmod{n}$ for any positive integer $a$ coprime to $n$. This result was proved earlier by Morrow \cite{morrow1951some} in the case $k=3$. In his proof, Makowski used the fact that there are infinitely many integers $n$ of the form $n = k\cdot p$ (where $p$ is a prime) such that $a^{n-k} \equiv 1 \pmod{n}$ for any integer $a$ if $\gcd(a,n) = 1$. Naturally, $\gcd(a,k) = 1$ for these numbers, and now the case for $\gcd(a,k) > 1$ remained open.

Rotkiewicz \cite{rotkiewicz1970pseudoprime, rotkiewicz1984congruence} showed that there are infinitely many positive integers $n$ such that $a^{n-k} \equiv 1 \pmod{n}$ if $k = 3$ and $a$ is an arbitrary fixed integer, or if $k=2$ and $a=2$. In 1987, Kiss and Phong \cite{kiss1987problem} gave a general solution for the case of $k>0$:
\begin{thm} \label{thm:kiss}
Let $a\geq 2$ and $k$ be fixed positive integers. Then there are infinitely many positive integers $n$ such that $a^{n-k} \equiv 1 \pmod{n}$.
\end{thm}
Our contribution to this problem is the extension of the work of Kiss and Phong to the case of $k \leq 0$, showing that for fixed integers $a \geq 2$ and $k$, there are infinitely many positive integers $n$ that $a^{n-k} \equiv 1 \pmod{n}$ if $(k, a) \neq (0,2)$. There are no integers $n>1$ satisfying $a^{n-k} \equiv 1 \pmod{n}$ if $(k,a) = (0,2)$.

The rest of the paper is structured as follows. Section \ref{sec:generalized} proves a generalization of Korselt's criterion, providing an equivalent condition for $n$ to satisfy the congruence condition in Question \ref{qn:main}. Section \ref{sec:special} looks at special cases of $k$ for Question \ref{qn:main}, in particular $k=0$, $k=1$, squarefree $k>0$, and squarefree $k<0$. Section \ref{sec:factorsnbyk} heuristically explains why a majority of the elements in $C_k$ share many factors with $k$. In Section \ref{sec:shortproducts}, we prove that there are finitely many integers in $C_k$ with one or two prime factors and finitely many elements in $C_k$ of the form $lpq$ where $l$ is a fixed integer and $p,q$ are primes, given certain restrictions. Section \ref{sec:conjCk} gives several conjectures on the growth rate of $C_k$, based on Erd\H{o}s's original conjecture on the growth rate of the Carmichael numbers \cite{erdos1956pseudoprimes}. Section \ref{sec:fixeda} builds on the work of Kiss and Phong to prove that for fixed integers $a \geq 2$ and $k$, there are infinitely many integers $n$ such that $a^{n-k} \equiv 1 \pmod{n}$ if and only if $(k,a) \neq (0,2)$. Finally, in Section \ref{sec:futurework}, we state a variant of Question \ref{qn:main} and prove some divisibility properties of the solutions to Question \ref{qn:notsqfree} that explain why for non-squarefree $k$, there are many integers $n$ such that $\lambda(n)$ divides $n-k$.

\section{Generalized Korselt's Criterion} \label{sec:generalized}

We prove a generalized version of Korselt's criterion.
\begin{prop} \label{prop:korselt}
An integer $n > \max(k,0)$ is in $C_k$ if and only if $n$ is squarefree and $\lambda(n)$ divides $n-k$.
\end{prop}
\begin{proof}
Suppose $n \in C_k$. By definition of the Carmichael function, there exists an element $a \in (\mathbb{Z}/n\mathbb{Z})^\times$ with order $\lambda(n)$. As $a$ is coprime to $n$, we have $a^{n-k} \equiv 1 \pmod{n}$. Then, the order of $a$ must divide the exponent on the left hand side, so $\lambda(n) \mid n-k$. Assume for the sake of contradiction that $n$ were not squarefree, so $p^2 \mid n$ for some prime $p$. Then we must have $p^{n-k+1} \equiv p \pmod{n}$, or $p^{n-k+1} \equiv p \pmod{p^2}$. Notice that $n-k+1 \geq 2$, so then $0 \equiv p \pmod{p^2}$. This is a contradiction.

Conversely, it is well-known that every element in $(\mathbb{Z}/n\mathbb{Z})^\times$ has order that divides $\lambda(n)$. Thus, if $\lambda(n) \mid n-k$, then $a^{n-k} \equiv 1 \pmod{n}$ holds for all $a$ coprime to $n$. 

Suppose $a$ is not coprime to $n$. We can write $a = ga'$ where $g = \gcd(a,n)$ and $\gcd(a', n) = 1$. From the previous paragraph, it follows that $(a')^{n-k+1} \equiv a' \pmod{n}$ so it suffices to prove that $g^{n-k+1} \equiv g \pmod{n}$. Notice that we can prime factorize $g=p_1 p_2 \cdots p_s$ and $n = p_1 p_2 \cdots p_t$ so that $p_i \neq p_j$ for all $i \neq j$ and $s \leq t$. This is a consequence of $g \mid n$ and $n$ being squarefree. It is clear that $g^{n-k+1} \equiv g \pmod{p_1p_2 \cdots p_s}$. Since $g$ is coprime to $p_{s+1} p_{s+2} \cdots p_t$ and $\lambda(p_{s+1} p_{s+2} \cdots p_t) \mid \lambda(n) \mid n-k$, we find that $g^{n-k+1} \equiv g \pmod{p_{s+1} p_{s+2} \cdots p_t}$. The Chinese Remainder Theorem implies that $g^{n-k+1} \equiv g \pmod{n}$ as desired.
\end{proof}
\begin{cor} \label{cor:generalizedkorselt}
An integer $n > \max(k,0)$ is in $C_k$ if and only if $n$ is squarefree and for all prime factors $p$ of $n$, $p-1$ divides $n-k$.
\end{cor}
\begin{proof}
This follows from the fact that $n \in C_k$ implies that $n$ is squarefree, and for squarefree $n$, $\lambda(n) = \lcm_{p \mid n}\{p-1\}$.
\end{proof}

\section{Special cases of \texorpdfstring{$k$}{k}} \label{sec:special}

\subsection{\texorpdfstring{$k=0$}{k = 0}}
Halbeisen and Hungerb\"{u}hler \cite{halbeisen1999generalised} proved that $C_0 = \{1,2,6,42,1806\}$.

\subsection{\texorpdfstring{$k=1$}{k = 1}}
By Fermat's little theorem and the definition of Carmichael numbers, $C_1$ consists of every prime and every Carmichael number. It has been shown that there are infinitely many Carmichael numbers \cite{alford1994there}, so $C_1$ contains infinitely many composite integers.

\subsection{Squarefree \texorpdfstring{$k > 0$}{k > 0}} \label{subsec:kgreater0}

It turns out most numbers in $C_k$ for a fixed, squarefree $k > 0$ follow a simple characterization, having $n = km$ for some positive integer $m$. As $n$ must be squarefree, we require $k$ to be squarefree and $m$ to be coprime to $k$.

Note that $\lambda(n) = \lcm(\lambda(k), \lambda(m))$. Thus, satisfying the generalized Korselt's criterion
is equivalent to satisfying two conditions:
\begin{equation} \label{eq:condition1}
    \lambda(k) \mid k(m-1)
\end{equation}
and
\begin{equation} \label{eq:condition2}
    \lambda(m) \mid k(m-1).
\end{equation}
Condition \eqref{eq:condition1} reduces to 
\[ \frac{\lambda(k)}{\gcd(\lambda(k), k)} \mid m-1, \]
which is a modulo condition on $m$. Condition \eqref{eq:condition2} can be better understood by noting that if $m$ is a prime or a Carmichael number, then the condition is satisfied because we have $\lambda(m) \mid m-1$. Then, having a factor of $k$ on the right side allows the right side to be divisible by $\lambda(m)$ more often. This allows us to define a more general type of Carmichael numbers: 

\begin{defn}
The \emph{$k$-Carmichael numbers} are positive, squarefree integers $m$ such that $\lambda(m) \mid k(m-1)$.
\end{defn}
\begin{rem}
The following are some basic properties of the $k$-Carmichael numbers.
\begin{enumerate}
    \item The $1$-Carmichael numbers consist of all of the primes and the Carmichael numbers. 
    \item If we have $r\mid s$ for some positive integers $r$ and $s$, then all $r$-Carmichael numbers are $s$-Carmichael numbers. 
    \item Every positive, squarefree integer $n$ is a $k$-Carmichael number for some $k > 0$.
\end{enumerate}
\end{rem}

Thus, for a fixed, squarefree $k > 0$, many of the positive integers $n$ such that $a^{n-k+1} \equiv a\pmod{n}$ for all integers $a$ are given by the following conditions:
\begin{itemize}
    \item $n = km$ for some positive, squarefree integer $m$ coprime to $k$.
    \item $m$ satisfies the modulo condition $m \equiv 1 \pmod{\frac{\lambda(k)}{\gcd(\lambda(k), k)}}$.
    \item $m$ is a $k$-Carmichael number. An easy way to satisfy this condition is to let $m$ be a prime number.
\end{itemize}
Dirichlet's prime number theorem tells us that there are infinitely many primes $m$ that satisfy the above three conditions. Thus, $\lvert C_k \rvert = \infty$ for squarefree $k > 0$.

\subsection{Squarefree \texorpdfstring{$k<0$}{k < 0}} \label{subsec:kless0}

We can apply a similar technique from the case of squarefree $k > 0$ to find a sufficient condition for the infinitude of $C_k$ for squarefree $k < 0$.
\begin{prop} \label{prop:kless0}
Let $k < 0$ be squarefree. If $n \in C_{-1}$, $n \equiv -1 \mod \left(\frac{\lambda(\lvert k \rvert)}{\gcd(\lambda(\lvert k \rvert), \lvert k \rvert)} \right)$, and $\gcd(n, \lvert k\rvert) = 1$, then $\lvert k \rvert n \in C_k$.
\end{prop}
\begin{proof}
The assumptions on $k$ and $n$ guarantee that $\lvert k \rvert n$ is squarefree. Now we use the generalized Korselt's criterion to show that $\lvert k \rvert n \in C_k$. We have
\begin{equation} \label{eq:modn}
    \lvert k \rvert n - k = \lvert k \rvert n + \lvert k \rvert = \lvert k \rvert(n+1) \equiv 0 \pmod{\lambda(n)}
\end{equation}
by our assumption that $n \in C_{-1}$. Moreover, the statement 
\begin{equation} \label{eq:modk-1}
    \lvert k \rvert n - k = \lvert k \rvert(n+1) \equiv 0 \pmod{\lambda(\lvert k \rvert)}
\end{equation}
is equivalent to a divisibility condition on $n+1$, namely $n \equiv -1 \mod \left(\frac{\lambda(\lvert k \rvert)}{\gcd(\lambda(\lvert k \rvert), \lvert k \rvert)}\right)$. As $\lvert k \rvert$ and $n$ are coprime, we have $\lcm(\lambda(\lvert k\rvert), \lambda(n)) = \lambda(\lvert k\rvert n)$ and so we can combine Equations \eqref{eq:modn}, \eqref{eq:modk-1} to obtain
\[ \lvert k \rvert n - k \equiv 0 \pmod{\lambda(\lvert k \rvert n)}, \]
which shows that $\lvert k \rvert n \in C_k$ by Proposition \ref{prop:korselt}.
\end{proof}

In 2012, Wright \cite{wright2012infinitely} proved that for any coprime $a,m \in \ZZ_{>0}$, there are infinitely many Carmichael numbers congruent to $a \pmod{m}$. If a similar result could be found for $C_{-1}$, then Proposition \ref{prop:kless0} would imply that $C_k$ is infinite for squarefree $k<0$ by picking $(a,m) = \left(-1, \frac{\lvert k \rvert \cdot \lambda(\lvert k \rvert)}{\gcd(\lambda(\lvert k \rvert), \lvert k \rvert)} \right)$. Here, $m$ is a multiple of $k$ to ensure that the guaranteed elements in $C_{-1}$ are coprime to $k$.

\section{Factors of \texorpdfstring{$n$}{n} based on \texorpdfstring{$k$}{k}} \label{sec:factorsnbyk}
In the previous sections, we saw that many elements in $C_{k}$ can be generated by multiplying $n \in C_{\pm 1}$ by $\lvert k \rvert$ under some modulo conditions on $n$.

Even if $n$ is not an integer multiple of $k$, it turns out that $n$ often shares factors with $k$, a pattern that is more visible with smaller factors. Suppose we are concerned whether $n$ and $k$ share a common divisor of $f$. If $n$ is divisible by a prime that congruent to $1 \pmod{f}$, then $f \mid \lambda(n)$. Proposition \ref{prop:korselt} implies that $k \equiv n \pmod{f}$. The hypothesis that $n$ is divisible by a prime that congruent to $1 \pmod{f}$ may appear to be rare at first, but as the size of $n$ increases, the average number of prime factors of $n$ increases as well. Then, it becomes rare for $n$ to have no prime factors that congruent to $1 \pmod{f}$. Asymptotically, $k \equiv n \pmod{f}$ holds. In particular, when $k$ has a factor of $f$, $n$ also has a factor of $f$. This heuristic is only broken when $n$ has only prime factors that not congruent to $1 \pmod{f}$.

\begin{ex}
    Many elements of $C_{-11}$ are congruent to $9 \pmod{10}$:
    \[C_{-11} = \{ \ldots, 28330\mathbf{9}, 30622\mathbf{9}, 31918\mathbf{9}, 33724\mathbf{9}, 35242\mathbf{9}, 38278\mathbf{9}, \ldots\}.\]
\end{ex}

The above heuristic is best observed for smaller factors $f$. Here, we consider $f=2, 4$ as they produce definitive restrictions on $n$.
\begin{prop}
For any $k$, let $n \in C_k$. If $n > 2$, then $k \equiv n \pmod{2}$.
\end{prop}
\begin{proof}
Since $n \in C_k$ is squarefree, if $n > 2$, then $n$ is divisible by some odd prime factor. This implies that $2 \mid \lambda(n)$, so by the generalized Korselt's criterion, $k \equiv n \pmod{2}$.
\end{proof}

\begin{prop}
Suppose that $4 \mid k$. Prime factors of elements in $C_k$ are either $2$ or congruent to $3 \pmod{4}$.
\end{prop}
\begin{proof}
If $n$ has a prime factor that is congruent to $1 \pmod{4}$, then $k \equiv n \pmod{4}$, or simply $4\mid n$. This means that $n$ is not squarefree and $n \notin C_k$.
\end{proof}
The consequence of the previous proposition is that the growth rate of $C_k$ is much smaller when $4 \mid k$ compared to when $4 \nmid k$. More precisely, the counting function $C_k(X) = \lvert C_k \cap (0, X] \vert$ grows slower whenever $k$ is a multiple of $4$.

We believe that there should be a way to quantify this heuristic argument to compare the growth rate of $C_k(X)$ across different values of $k$.

\section{Short Products in \texorpdfstring{$C_k$}{C\_k}} \label{sec:shortproducts}
In this section, we investigate elements in $n \in C_k$ that are the products of few primes. We also show that if all but two prime factors of $n$ are fixed, then there are finitely many such $n$, excluding certain exceptions.

\subsection{One prime factor}

Halbeisen and Hunger\"{u}hler \cite[Proposition~4.1]{halbeisen1999generalised} showed that for any $k$, $C_k$ contains finitely many primes.

\subsection{Products of two primes}
Halbeisen and Hungerb\"{u}hler \cite[Proposition~4.4]{halbeisen1999generalised} determined that for any $k$, there are finitely many elements in $C_k$ with exactly two prime factors unless $k>0$ is prime. Here, we provide an alternate proof.
\begin{thm}
There are infinitely many integers in $C_k$ with exactly two prime factors if and only if $k>0$ is prime. In fact,
\begin{enumerate}
    \item For any $k$, there are finitely many elements in $C_k$ with exactly two prime factors both coprime to $k$. \label{itm:pq1}
    \item If $k < 0$, or $k>0$ is not prime, there are finitely many integers in $C_k$ with exactly two prime factors, at least one of which divides $k$. \label{itm:pq2}
\end{enumerate}
\end{thm}
\begin{proof}
First, we prove statement \eqref{itm:pq1}. Suppose $pq \in C_k$ for distinct primes $p,q$ both coprime to $k$. By the generalized Korselt's criterion we obtain
\[ p-1 \mid pq - k, \quad q-1 \mid pq - k. \]
As $p \equiv 1 \pmod{p-1}$ and $q \equiv 1 \pmod{q-1}$, we get 
\[ p-1 \mid q - k, \quad q-1 \mid p - k. \]
The condition that $p,q$ are coprime to $k$ guarantees that the right side of each equation is nonzero. Thus, we obtain inequalities bounding the magnitude of each side:
\[ \lvert p-1 \rvert \leq \lvert q-k \rvert, \quad \lvert q-1 \rvert \leq \lvert p-k \rvert. \]
Graphing the intersection of the above inequalities reveals that for $k >0$, we must have $p+q \leq k+1$, thus there are finitely many integers in $C_k$ with two prime factors both coprime to $k$ if $k>0$.

For $k < 0$, we obtain $\lvert p-q \rvert \leq 1-k$. Without loss of generality, let $p < q$ to get
\begin{equation} \label{eq:pqk1}
p \geq q + k - 1.
\end{equation}
From $p-1 \mid q-k$, we obtain $\alpha(p-1) = q-k$ for some integer $\alpha$. As $p < q$ and $k \leq -1$, we must have $\alpha \geq 2$, or $2(p-1) \leq q-k$. Thus we obtain
\[ q \geq 2p + k - 2. \]
Solving this with Equation \eqref{eq:pqk1} gives $p < q \leq 4 - 3k$. Thus, there are only finitely many integers in $C_k$ with exactly two prime factors both coprime to $k$.

Now we prove statement \eqref{itm:pq2}. For primes $p,q$, let $pq \in C_k$. Without loss of generality, let $p \mid k$. Notice that there are finitely many possibles values of $p$. By Proposition \ref{prop:korselt} we obtain $q-1 \mid pq - k$. As $q \equiv 1 \pmod{q-1}$, we have $q-1 \mid p-k$. If $k<0$, or $k>0$ and $k$ is not prime, then $p-k \neq 0$, so $\lvert q-1 \rvert \leq \lvert p-k \rvert$. This bounds the value of $q$, which completes the proof of \eqref{itm:pq2}.

Noting that $C_0$ is finite, we have proved that there are finitely many integers in $C_k$ with exactly two prime factors unless $k>0$ is prime. In Section \ref{sec:special}, we showed that if $k>0$ is prime, then every integer in the form $n=k\cdot m$ for a prime $m \equiv 1 \pmod{\frac{\lambda(k)}{\gcd(\lambda(k), k)}}$ with $\gcd(k,m) = 1$ is in $C_k$. Thus, there are infinitely many integers in $C_k$ with exactly two prime factors if and only if $k>0$ and $k$ is prime. 
\end{proof}

\subsection{Products of three or more primes}

It is well-known that if $6n+1, 12n+1,$ and $18n+1$ are primes, then $(6n+1)(12n+1)(18n+1)$ is a Carmichael number. Given this fact, it may be plausible to conjecture that there are infinitely many numbers in $C_k$ which are products of three or more primes. However, we can prove that there are finitely many numbers in $C_k$ if we fix all but two of the primes (excluding some special cases). More precisely,
\begin{prop} \label{prop:lpq}
    For any integers $k$ and $l>0$, there are finitely many primes $p,q$ such that $lpq \in C_k$, $l \neq k$, $lp \neq k$, and $lq \neq k$.
\end{prop}
\begin{proof}
As $lpq$ must be squarefree, we can assume that $l,p,q$ are pairwise coprime. Using Proposition \ref{prop:korselt} gives
\[ p-1 \mid lpq - k, \quad q-1 \mid lpq - k.\]
As $p \equiv 1\pmod{p-1}$ and $q \equiv 1 \pmod{q-1}$, we obtain
\[ p-1 \mid lq - k, \quad q-1 \mid lp - k.\]
Let $\alpha, \beta$ be integers such that
\begin{equation} \label{eq:lpqalphabeta}
     \alpha(p-1) = lq-k, \quad \beta(q-1) = lp - k.
\end{equation}
If $\alpha = 0$ or $\beta =0$, then $lq = k$ and $lp=k$, respectively, so we can just consider $\alpha, \beta \neq 0$.

Solve for $p$ in the first equation to get $p = \frac{lq-k}{\alpha} + 1$. Substitute this into the second equation to obtain
\begin{align*}
    \beta(q - 1) &= l\left(\frac{lq-k}{\alpha} + 1\right) - k, \\
    \alpha\beta(q - 1) &= l^2q - lk + \alpha l - \alpha k, \\
    q (\alpha \beta - l^2) &= \alpha\beta - lk + \alpha l - \alpha k.
\end{align*}

Suppose $\alpha \beta - l^2 = 0$. Then, $l^2 - lk + \alpha l - \alpha k = (l-k)(\alpha+l) = 0$. We may ignore $l=k$, so we get $\alpha = -l$ and similarly $\beta = -l$. Adding the two equations in \eqref{eq:lpqalphabeta} then gives $l(p+q-1) = k$. There are finitely many primes satisfying this equation with $p,q \geq 2$. Thus, we may assume $\alpha \beta - l^2 \neq 0$.

Solving for $q$:
\begin{align*}
    q &= \frac{\alpha\beta - lk + \alpha l - \alpha k}{\alpha\beta - l^2}, \\
    q &= 1 + \frac{l^2 - lk + \alpha l - \alpha k}{\alpha\beta - l^2}, \\
    q &= 1 + \frac{\frac{l^2 - lk}{\alpha} +l - k}{\beta - \frac{l^2}{\alpha}}.
\end{align*}

We claim that for fixed $l,k$, the final fraction has an upper bound for parameters $\alpha, \beta \in \ZZ$ as long as $\alpha\beta \neq l^2$. The magnitude of the numerator is bounded above by $\lvert l^2 \pm lk + l \pm k\rvert$. For the denominator, note that for $\lvert \alpha \rvert \geq 2l^2$, we have $-\frac{-1}{2} \leq \frac{l^2}{\alpha} \leq \frac{1}{2}$. Thus, for $\beta \geq 1$,
\[ \left\lvert \beta - \frac{l^2}{\alpha} \right\rvert = \beta - \frac{l^2}{\alpha} \geq 1 - \frac{1}{2} = \frac{1}{2}, \] 
and for $\beta \leq -1$,
\[ \left\lvert \beta - \frac{l^2}{\alpha} \right\rvert = -\left(\beta - \frac{l^2}{\alpha}\right) \geq 1 - \frac{1}{2} = \frac{1}{2}. \]
This proves the lower bound of the magnitude of the denominator for $\lvert \alpha \rvert \geq 2l^2$. 

For each $\alpha$ in the range $\lvert \alpha \rvert < 2l^2$, the denominator is nonzero, so we may lower bound magnitude of the denominator by a positive constant. Thus, the fraction is bounded above, and so is $q$. By symmetry, $p$ is also bounded above. This proves that there are finitely many primes $p,q$ such that $lpq \in C_k$, $l \neq k$, $lp \neq k$, and $lq \neq k$.
\end{proof}

In 1993, Pinch \cite{pinch1993carmichael} gave bounds on $\alpha$ and $\beta$ for $k=1$ to formulate an algorithm to enumerate the Carmichael numbers up to $10^{15}$. The following bounds that we derive for $\alpha$ and $\beta$ for the other values of $k$ can be used to modify Pinch's algorithm for $C_k$.

\begin{prop}
Let $P$ denote the largest prime factor of $l$ and assume that $P < p < q$ where $p,q$ are primes. If $lpq \in C_k$, $l \neq k$, $lp \neq k$, and $lq \neq k$, then there exist integers $\alpha, \beta$ such that
\[ p = 1 + \frac{(l-k)(l + \beta)}{\alpha\beta - l^2}, \]
\[ q = 1 + \frac{(l-k)(l + \alpha)}{\alpha\beta - l^2}. \]
If $l > k > 0$, then
\[ 0 < \alpha \beta \leq l^2 \left(\frac{P + 3}{P + 1}\right). \]
\end{prop}
\begin{proof}
The existence of $\alpha$ and $\beta$ was shown in the proof of Proposition \ref{prop:lpq}.

Since $\frac{q-2}{p-1} \geq 1$ and $\frac{p}{q-1} \leq 1$, we may bound $\alpha$ and $\beta$:
\[ \alpha = \frac{lq-k}{p-1} = l\cdot \frac{q-2}{p-1} + \frac{2l-k}{p-1} \geq l + \frac{2l -k}{p-1},\]
\[ \beta = \frac{lp-k}{q-1} = l\cdot \frac{p}{q-1} - \frac{k}{q-1} \leq l - \frac{k}{q-1}.\]
From the proof of Proposition \ref{prop:lpq}, we have
\begin{align*}
	(\alpha \beta - l^2)(q-1) &= (l-k)(\alpha + l) \\
   		&\geq (l-k)\left(2l + \frac{2l-k}{p-1}\right) \\
        &\geq (l-k)\left(\frac{2lp-k}{p-1}\right).
\end{align*}
Thus, if $l > k$, then $\alpha \beta - l^2 > 0$.

In addition, we have
\[ P + 1 \leq p-1 = \frac{(l-k)(\beta + l)}{\alpha \beta - l^2} \leq \frac{(l-k)\left(2l - \frac{k}{q-1}\right)}{\alpha \beta - l^2}. \]
Thus,
\[\alpha\beta - l^2 \leq \frac{(l-k)\left(2l - \frac{k}{q-1}\right)}{P+1}, \]
or simplified more,
\[ \alpha\beta\leq l^2 \left(\frac{P + 1 + (1-\frac{k}{l})\left(2 - \frac{k}{l(q-1)}\right)}{P+1} \right). \]
The simplest case to consider is $k > 0$, in which $(1-\frac{k}{l})\left(2 - \frac{k}{l(q-1)}\right) \leq 2$. Thus, if $l > k > 0$, then
\[ 0 < \alpha \beta \leq l^2 \cdot\frac{P + 3}{P + 1}. \]
\end{proof}

We can also prove that there are finitely many integers $n \in C_k$ satisfying a particular divisibility condition on their prime factors.

\begin{prop}
If $k$ has at least $3$ prime factors (counted with multiplicity), there are finitely many primes $p,q,r$ such that $pqr \in C_k$, $p-1 \mid q-1$, and $p-1 \mid r-1$.
\end{prop}
\begin{proof}
Let $pqr \in C_k$ for primes $p,q,r$ such that $p-1 \mid q-1$ and $p-1 \mid r-1$. From the generalized Korselt's criterion, we get $p-1 \mid pqr - k$. As $p \equiv q \equiv r \equiv 1 \pmod{p-1}$, this simplifies to $p-1 \mid k$. This means that there are finitely many possible values for $p$. Because $p$, $pq$, and $pr$ have no chance of being equal to $k$, Proposition \ref{prop:lpq} shows that for each of those values of $p$, there are finitely many possible values of $q$ and $r$ such that $pqr \in C_k$. This completes the proof.
\end{proof}

\section{Conjectures on the Growth Rate of \texorpdfstring{$C_k$}{C\_k}} \label{sec:conjCk}

In this section, we state several conjectures on the growth rate of the counting functions
\[ C_k(X) = \lvert C_k \cap (0, X] \rvert \]
and
\[ N_k(X) = \lvert C_k \cap (0, X] \rvert \]
where $N_k$ is the set of generalized Carmichael numbers from Definition \ref{defn:generalizedcarmichael}.

The prime number theorem on arithmetic progressions states that for coprime integers $a$ and $m$, the number of primes congruent to $a \pmod{m}$ less than $X$ is asymptotically $\phi(m)^{-1} \frac{X}{\log X}$ where $\phi$ is Euler's totient function.

Thus, for squarefree $k > 0$, the counting function $N_k(X) = N_k \cap (1, X]$ is asymptotically
\[ C_k(X) - \phi\left(\frac{\lambda(k)}{\gcd(\lambda(k), k)}\right)^{-1} \cdot \frac{X/k}{\log(X/k)}. \]
For the other values of $k$, $N_k(X) = C_k(X)$ by definition.

Figure \ref{fig:Nvsk} shows $N_k(10^7)$ for squarefree $2 \leq k \leq 1000$. Figure \ref{fig:Nvs-k} shows $N_{-k}(10^7)$ for squarefree $2 \leq k \leq 1000$. Figure \ref{fig:NvsN} shows $N_{k}(10^7)$ versus $N_{-k}(10^7)$ for squarefree $2 \leq k \leq 1000$; note that the horizontal axis is $N_k(10^7)$ and the vertical axis is $N_{-k}(10^7)$. Each data point is color-coded based on the number of prime factors of $k$.

\begin{figure}
    \centering
    \includegraphics[width=\linewidth]{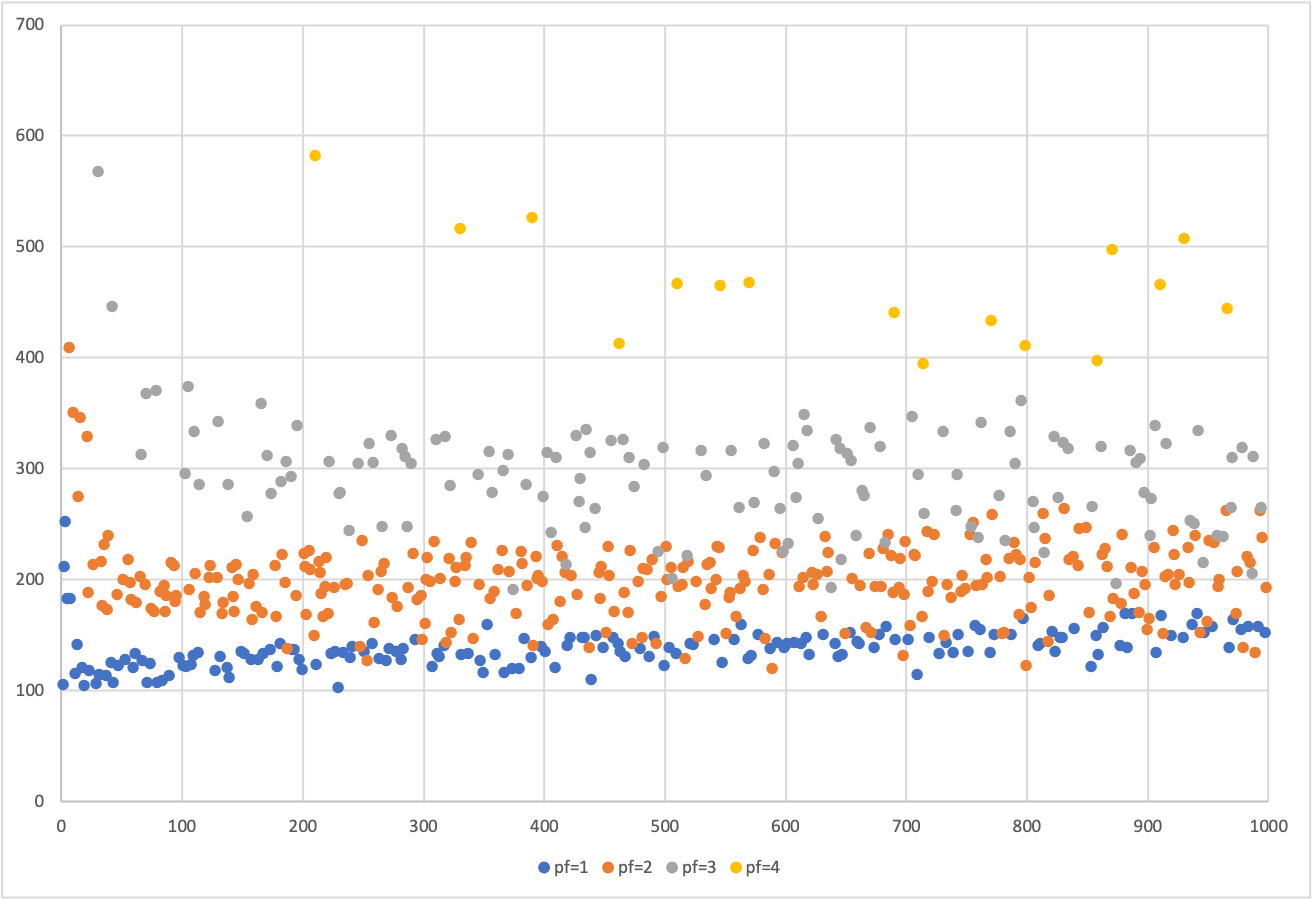}
    \caption{$N_k(10^7)$ for squarefree $2 \leq k \leq 1000$.}
    \label{fig:Nvsk}
\end{figure}
\begin{figure}
    \centering
    \includegraphics[width=\linewidth]{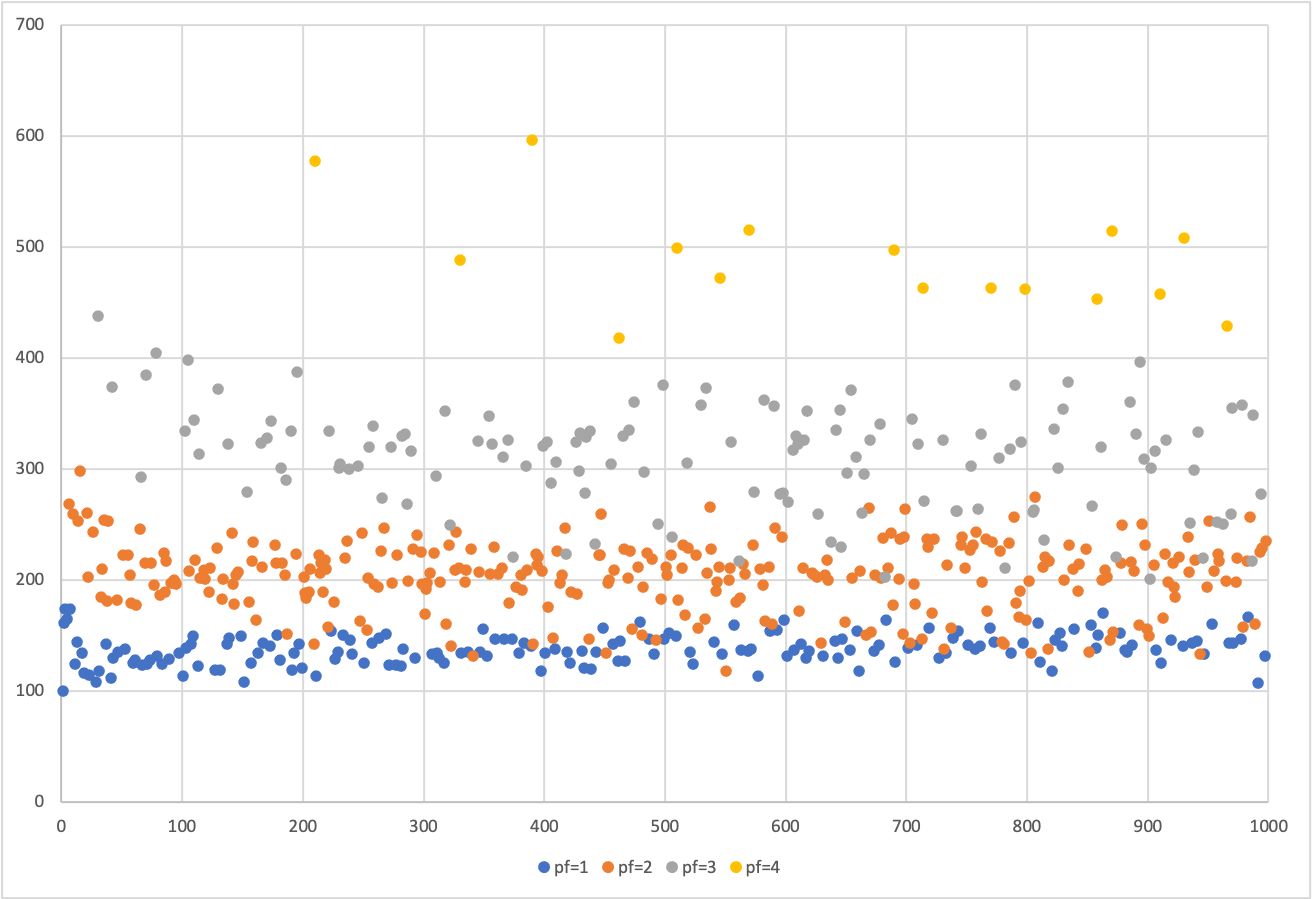}
    \caption{$N_{-k}(10^7)$ for squarefree $2 \leq k \leq 1000$.}
    \label{fig:Nvs-k}
\end{figure}
\begin{figure}
    \centering
    \includegraphics[width=\linewidth]{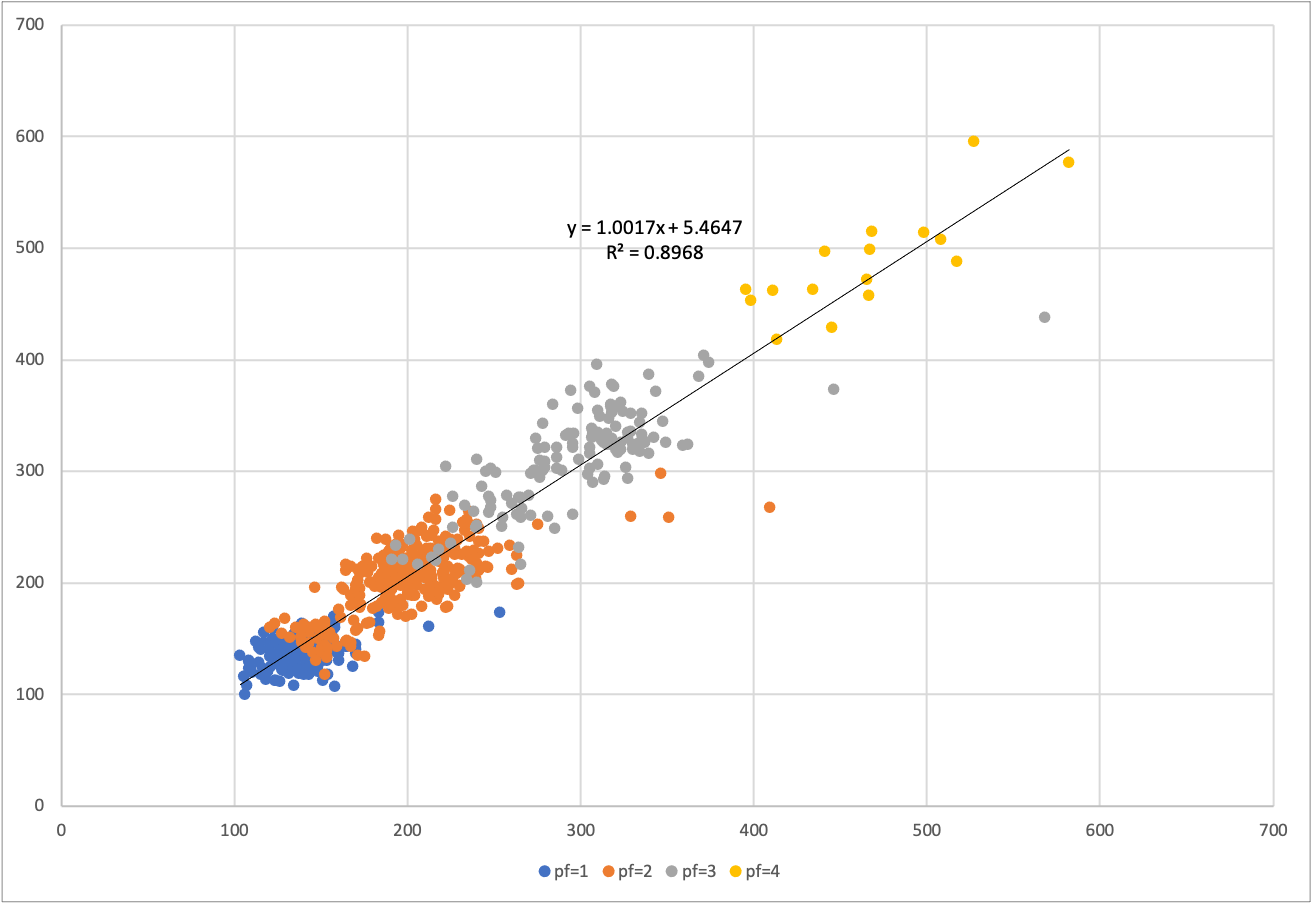} 
    \caption{$N_{-k}(10^7)$ vs. $N_{k}(10^7)$ for squarefree $2 \leq k \leq 1000$.}
    \label{fig:NvsN}
\end{figure}

Figures \ref{fig:Nvsk} and \ref{fig:Nvs-k} illustrate that for the same value of $X$, $N_k(X)$ and $N_{-k}(X)$ tend to be bigger if $k$ has many prime factors.

Moreover, Figure \ref{fig:NvsN} shows that $N_k(X)$ and $N_{-k}(X)$ seem to grow at similar rates with respect to $X$. Our tests show that $N_k(X)$ and $N_{-k}(X)$ are similar even when $k$ is non-squarefree; unfortunately, this is hard to show because $N_k(X)$ is small when $k$ is non-squarefree. Nevertheless, we conjecture (Conjecture \ref{conj:NkvN-k}) that for all integers $k$,
\[ \lim_{X \to \infty} \frac{N_{-k}(X)}{N_k(X)} = 1. \]

In Figure \ref{fig:dvsd}, we graph the function $d_k(10^7)$ against $d_{-k}(10^7)$ ($x$-axis and $y$-axis, respectively) where $d_k(X)$ is the function defined in Definition \ref{defn:dcoefficient}. Re-expressing $N_k(X)$ as $d_k(X)$ seems to re-scale the data such that the points are more uniform along the regression line. Moreover, the re-scaling seems to clarify the existence of four clusters of data points, formed by distinguishing the number of primes factors of $k$.

\begin{figure}
    \centering
    \includegraphics[width=\linewidth]{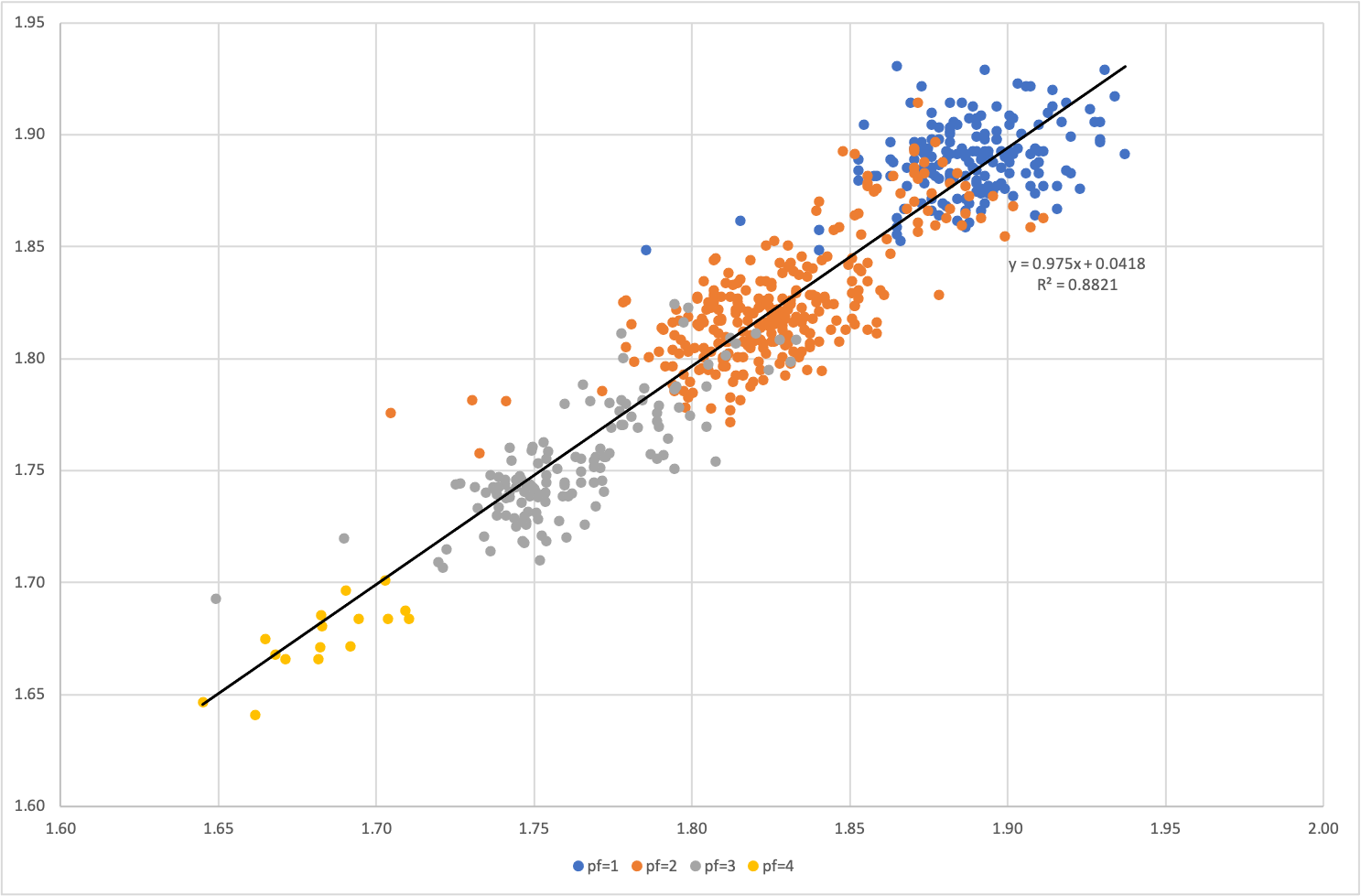}
    \caption{$d_{-k}(10^7)$ vs. $d_{k}(10^7)$ for squarefree $2 \leq k \leq 1000$.}
    \label{fig:dvsd}
\end{figure}
	    
Unfortunately, although $d_k(10^7)$ is generally smaller when $k$ has many prime factors, it is not clear what the limit $\lim_{X \to \infty} d_k(X)$ may look like with respect to $k$, if it exists at all.

\section{For fixed base \texorpdfstring{$a$}{a}} \label{sec:fixeda}
In this section, we extend the Kiss and Phong's proof of Theorem \ref{thm:kiss} to prove that for fixed integers $a \geq 2$ and $k$, there are infinitely many integers $n$ such that $a^{n-k} \equiv 1 \pmod{n}$ if and only if $(k, a) \neq (0,2)$. First, we note that the following key lemma of Kiss and Phong \cite{kiss1987problem}:

\begin{lemma} \label{lemma:kiss}
Let $a,k,m$ be positive integers satisfying $a > 1$, $m-k>1$, and $\gcd(a,m) = 1$. If $\ord_m(a) \mid m-k$ but $\ord_m(a) < m-k$, then there are infinitely many positive integers $n$ that satisfy $a^{n-k} \equiv 1 \pmod{n}$, unless $m-k=2$ and $a+1$ is a power of $2$, or $m-k = 6$ and $a=2$.
\end{lemma}

This lemma generalizes for all integers $k$. Moreover, the same proof holds since the argument is purely algebraic.

\begin{lemma} \label{lemma:generalkiss}
Let $a,k,m$ be integers satisfying $a > 1$, $m-k>1$, $m \geq 1$, and $\gcd(a,m) = 1$. If $\ord_m(a) \mid m-k$ but $\ord_m(a) < m-k$, then there are infinitely many positive integers $n$ that satisfy $a^{n-k} \equiv 1 \pmod{n}$, unless $m-k=2$ and $a+1$ is a power of $2$, or $m-k = 6$ and $a=2$.
\end{lemma}

This enables us to extend Theorem \ref{thm:kiss} for all integers $k$.

\begin{thm}
If $a\geq 2$ and $k$ are integers with $(k,a) \neq (0,2)$, there are infinitely many positive integers $n$ such that $a^{n-k} \equiv 1 \pmod{n}$. If $(k,a) = (0,2)$, then there are no integers $n>1$ such that $a^{n-k} \equiv 1 \pmod{n}$.
\end{thm}
\begin{proof}
Theorem \ref{thm:kiss} completely solves the case $k > 0$. Thus, we may assume that $k\leq 0$. For each pair of values $(k,a)$, it is sufficient to find an integer $m$ such that it satisfies the conditions of Lemma \ref{lemma:generalkiss}. For the following values of $k$ and $a$, letting $m = a-1$ works:
\begin{itemize}
    \item $k=0$ and $a \geq 4$
    \item $-1 \geq k \geq -4$ and $a \geq 2$
    \item $k=-5$ and $a \geq 3$
    \item $k \leq -6$ and $a \geq 2$
\end{itemize}
In all of these cases, $m-k > 1$, and if $m-k=6$, $a > 2$. Moreover, $\ord_m(a) = 1$, so $\ord_m(a) \mid m-k$ and $\ord_m(a) <  m-k$ hold automatically. Thus, we can use Lemma \ref{lemma:generalkiss}.

If we try to set $m=a-1$ for $(k,a) = (0,2)$, then we have $m-k = 1$, which is not allowed. If we set $m=a-1$ for $(k,a) = (0,3)$, then $m-k=2$ and $a+1 = 4$, which is also not allowed. If we set $m=a-1$ for $(k,a) = (-5, 2)$, then $m-k = 6$ and $a=2$, which is not allowed.

When $(k,a) = (0,3)$, we can apply Lemma \ref{lemma:generalkiss} with $m=8$, since $\ord_m(a) = 2 \mid 8 = m-k$. When $(k,a) = (-5,2)$, we let $m = 3$ so that $\ord_m(a) = 2 \mid 8 = m-k$.

In the case $(k,a) = (0,2)$, we obtain the congruence relation $2^n \equiv 1 \pmod{n}$. For the sake of contradiction, suppose that $n>1$ is an integer such that $2^n \equiv 1 \pmod{n}$. Let $p$ be the smallest prime divisor of $n$ so that $2^n \equiv 1 \pmod{p}$. Hence, $\ord_p(2) \mid n$. However, $\ord_p(2) < p$ and since $p$ was the smallest prime divisor $n$, $\ord_p(2) = 1$. That is, $p$ divides $2^1 - 1 = 1$, which is impossible.
\end{proof}

\section{Future Work} \label{sec:futurework}

We hope to heuristically explain our observations of Figure \ref{fig:dvsd} from Section \ref{sec:conjCk} and suggest reasonable guesses for the limit $\lim d_k(X)$. In addition, we note that some points from the ``pf=2'' category seem to be in the ``pf=1'' cluster, and similarly some points from the ``pf=3'' category seem to be in the ``pf=2'' cluster. 

We are unsure why this phenomenon occurs; however, it is plausible that $d_k(X)$ is proportional to a function of $k$ that is generally small when $k$ has many prime factors, but there are subtleties based on certain divisibility conditions on the prime factors. For example, $\lambda(k)$ is generally small when $k$ has many prime factors, but the actual size depends on the $\lcm$ of $p-1$ for all prime factors $p$ of $k$.

We may be able to learn more about $C_k$ by answering a variant of our original problem.
\begin{qn}\label{qn:notsqfree}
Given an integer $k$, for what integers $n > \max(k, 0)$ is $a^{n-k+1} \equiv a \mod{n}$ for all integers $a$ coprime to $n$?
\end{qn}
\begin{notation}
We denote the set of positive integers $n$ that satisfy the condition in Question \ref{qn:notsqfree} by $C'_k$:
\[ C'_k = \{ n \in \ZZ : n > \max(k,0) \text{ and } a^{n-k+1} \equiv a \pmod{n} \text{ for all integers } a \text{ coprime to $n$}\}. \]
\end{notation}
There is a simple condition to check that an integer is in $C'_k$.
\begin{prop}
    An integer $n > \max(k,0)$ is in $C'_k$ if and only if $\lambda(n) \mid n-k$.
\end{prop}
\begin{proof}
    The proof is similar to that of Proposition \ref{prop:korselt}.
\end{proof}

Note that for all integers $k$, $C_k \subseteq C'_k$. The key difference between $C_k$ and $C'_k$ is that the integers in $C_k$ must be squarefree while the integers in $C'_k$ may be non-squarefree.

It turns out that $C_1 = C'_1$. The \emph{composite} integers in $C'_k$ for $k>0$ are called the \emph{$k$-Kn\"{o}del numbers}, named after Walter Kn\"{o}del. In 1963, Makowski \cite{makowski1962generalization} proved that for any $k \geq 2$, there are infinitely many composite integers $n$ such that $a^{n+k-1} \equiv a \pmod{n}$ for all integers $a$ coprime to $n$.

\begin{figure}
    \centering
    \includegraphics[width=\linewidth]{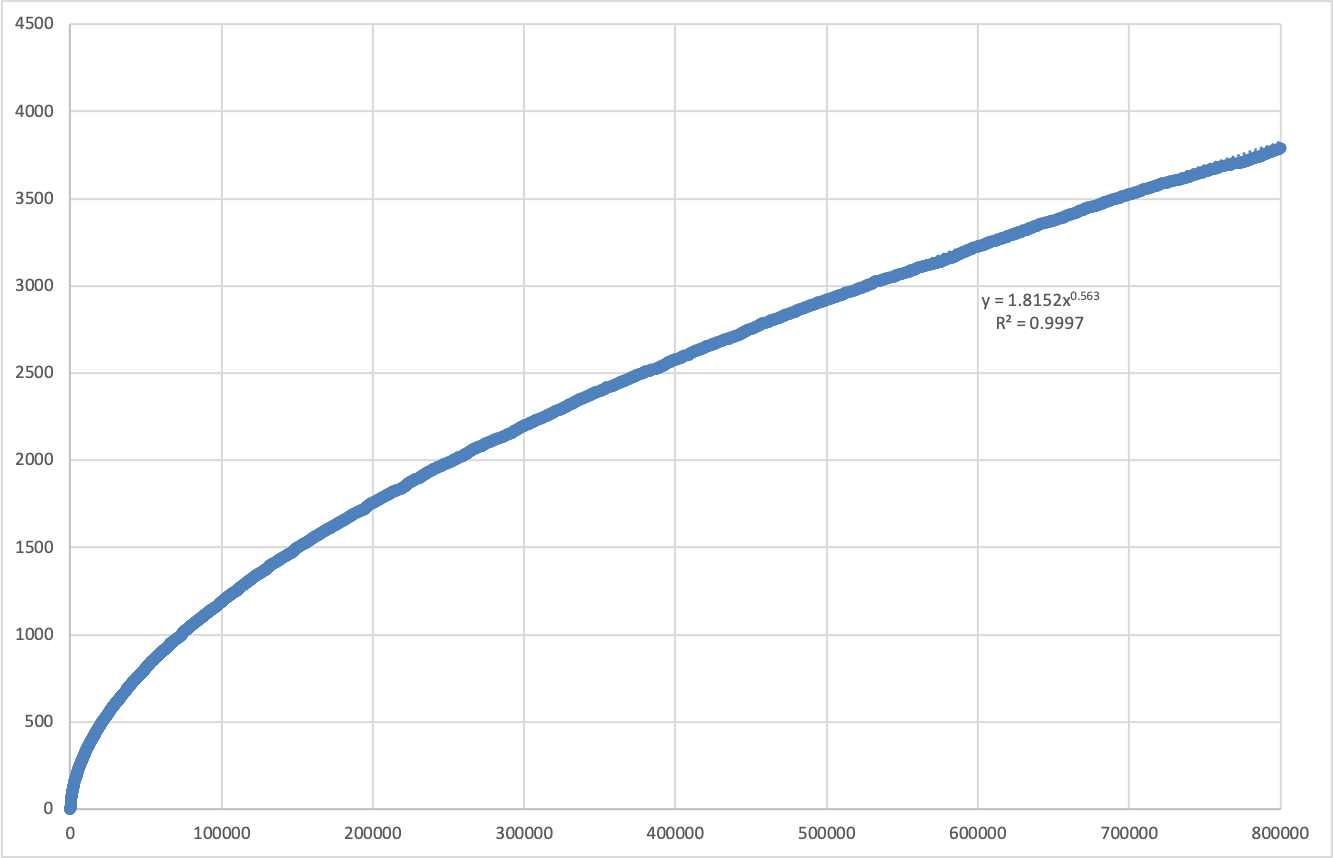}
    \caption{$C'_0(X)$ for $1 \leq X \leq 8 \times 10^5$.}
    \label{fig:C0powerfit}
\end{figure}

We numerically observe that $C'_k$ is large for the non-squarefree $k$, which contrasts the pattern that $C_k$ is small for non-squarefree $k$. In particular, there are many elements in $C'_0$. In Figure \ref{fig:C0powerfit}, we graph the counting function $C'_0(X) = \lvert C'_0 \cap (0, X] \rvert$. The following two propositions may explain this phenomenon.

First, we note that the exponents in the prime factorization of $n \in C'_k$ are bounded by $k$:
\begin{prop} \label{prop:notsqfreepower}
    If $\lambda(n)$ divides $n-k$ and $n = \prod_{i=1}^r p_i^{e_i}$, then $\prod_{i=1}^r p_i^{e_i-1}$ divides $k$.
\end{prop}
Moreover, elements from $C_k$ can generate elements in $C_{kp}$ for prime factors $p$ of $k$:
\begin{prop}\label{prop:notsqfreegenerate}
If $n \in C'_k$ is a multiple of a prime $p$, then $np \in C'_{kp}$.
\end{prop}
For $k=0$, Proposition \ref{prop:notsqfreepower} says that the exponents of the prime powers that divide $n$ are unbounded. Proposition \ref{prop:notsqfreegenerate} implies that if $n \in C'_0$ is a multiple of a prime $p$, then $np \in C'_0$. Quantifying this property may explain the exact growth rate of $C'_0(X)$.

Moreover, Proposition \ref{prop:notsqfreegenerate} suggests two categorizations of elements in $n \in C'_k$: ``old'' elements for which there exists $d \mid \gcd(n,k)$ with $d > 1$ such that $\frac{n}{d} \in C'_{k/d}$, and ``new'' elements without any such $d$. Further analysis of these types of elements may lead to a heuristic explanation for the largeness of $C'_k(X)$ for non-squarefree $k$.

\section{Acknowledgements}
Special thanks to Stefan Wehmeier for suggesting the project and providing advice on the best direction for research. We would like to thank the MIT PRIMES program for the opportunity to perform research under the mentorship of wonderful professors and graduate students.

\bibliography{citations}

\begin{thebibliography}{PSW80}

\bibitem[AGP94]{alford1994there}
W.~R. Alford, Andrew Granville, and Carl Pomerance.
\newblock There are infinitely many carmichael numbers.
\newblock {\em Ann. Math.}, pages 703--722, 1994.

\bibitem[Car10]{carmichael1910note}
R.~D. Carmichael.
\newblock Note on a new number theory function.
\newblock {\em Bull. Am. Math. Soc.}, 16(5):232--238, 1910.

\bibitem[Car12]{carmichael1912composite}
R.~D. Carmichael.
\newblock On composite numbers $p$ which satisfy the fermat congruence $a^{p-1}
  \equiv 1 \pmod{p}$.
\newblock {\em Amer. Math. Monthly}, 19(2):22--27, 1912.

\bibitem[Che39]{chernick1939fermat}
Jack Chernick.
\newblock On fermat's simple theorem.
\newblock {\em Bull. Amer. Math. Soc.}, 45(4):269--274, 1939.

\bibitem[Erd56]{erdos1956pseudoprimes}
Paul Erd\H{o}s.
\newblock On pseudoprimes and carmichael numbers.
\newblock {\em Publ. Math. Debrecen}, 4(1956):201--206, 1956.

\bibitem[HH99]{halbeisen1999generalised}
L.~Halbeisen and N.~Hungerb\"{u}hler.
\newblock On generalised carmichael numbers.
\newblock {\em Hardy-Ramanujan J.}, 1999.

\bibitem[Kor99]{korselt1899probleme}
Alwin Korselt.
\newblock Probleme chinois.
\newblock {\em L’interm{\'e}d. Math.}, 6:143--143, 1899.

\bibitem[KP87]{kiss1987problem}
P{\'e}ter Kiss and Bui~Minh Phong.
\newblock On a problem of a. rotkiewicz.
\newblock {\em Math. Comp.}, pages 751--755, 1987.

\bibitem[Mak63]{makowski1962generalization}
A.~Makowski.
\newblock Generalization of morrow's d-numbers.
\newblock {\em Bull. Belg. Math. Soc. Simon Stevin}, 36:71, 1962/1963.

\bibitem[Mor51]{morrow1951some}
DC~Morrow.
\newblock Some properties of d numbers.
\newblock {\em Am. Math. Monthly}, 58(5):329--330, 1951.

\bibitem[Pin93]{pinch1993carmichael}
R.~G.~E. Pinch.
\newblock The carmichael numbers up to $10^{15}$.
\newblock {\em Math. Comp.}, 61(203):381--391, 1993.

\bibitem[Pom81]{pomerance1981distribution}
Carl Pomerance.
\newblock On the distribution of pseudoprimes.
\newblock {\em Math. Comp.}, pages 587--593, 1981.

\bibitem[Pom89]{pomerance1989two}
Carl Pomerance.
\newblock Two methods in elementary analytic number theory.
\newblock {\em Number theory and applications}, 265:135--161, 1989.

\bibitem[PSW80]{pomerance1980pseudoprimes}
Carl Pomerance, J.~L. Selfridge, and Samuel~S. Wagstaff.
\newblock The pseudoprimes to $25\cdot 10^9$.
\newblock {\em Math. Comp.}, 35(151):1003--1026, 1980.

\bibitem[Rot70]{rotkiewicz1970pseudoprime}
Andrzej Rotkiewicz.
\newblock {\em Pseudoprime numbers and their generalizations}.
\newblock University of Novi Sad, 1970.

\bibitem[Rot84]{rotkiewicz1984congruence}
A.~Rotkiewicz.
\newblock On the congruence $2^{n-2} \equiv 1 \pmod{n}$.
\newblock {\em Math. Comp.}, pages 271--272, 1984.

\bibitem[Wri12]{wright2012infinitely}
Thomas Wright.
\newblock Infinitely many carmichael numbers in arithmetic progressions.
\newblock {\em arXiv preprint arXiv:1212.5850}, 2012.

\end{thebibliography}
\bibliographystyle{halpha}
\end{document}